\documentclass[11pt,a4paper,reqno]{amsart}
\usepackage{amsmath,amscd,amssymb,latexsym}
\usepackage{hyperref}
\usepackage{indentfirst}
\usepackage{mathtools}
\usepackage{enumitem}
\usepackage{float}
\usepackage{multirow}
\usepackage{makecell}
\usepackage{xypic}
\usepackage{todonotes}
\usepackage{algorithmic,algorithm}
\usepackage{kotex}

\usepackage[last]{changes}

\newtheorem{Thm}{Theorem}[section]
\newtheorem{Cor}[Thm]{Corollary}
\newtheorem{Lem}[Thm]{Lemma}

\theoremstyle{definition}
\newtheorem{Def}[Thm]{Definition}

\theoremstyle{remark}
\newtheorem{Rk}[Thm]{Remark}
\newtheorem{Ex}[Thm]{Example}
\newtheorem{Pb}[Thm]{Problem}

\def\N{{\mathbb N}}

\def\Z{{\mathbb Z}}

\newtheorem{theorem}{Theorem}[section]
\newtheorem{lemma}[theorem]{Lemma}

\theoremstyle{definition}
\newtheorem{definition}[theorem]{Definition}

\newtheorem{corollary}[theorem]{Corollary}
\newtheorem{proposition}[theorem]{Proposition}

\theoremstyle{remark}
\newtheorem{remark}[theorem]{Remark}

\numberwithin{equation}{section}

\begin{document}

\title{The Frobenius problem for Numerical Semigroups generated by binomial coefficients}


\author{WonTae Hwang}
\address{Department of Mathematics and Institute of Pure and Applied Mathematics, Jeonbuk National University, 567 Baekje-daero, Deokjin-gu, Jeonju-si, Jeollabuk-do, 54896, South Korea}
\curraddr{Department of Mathematics and Institute of Pure and Applied Mathematics, Jeonbuk National University, 567 Baekje-daero, Deokjin-gu, Jeonju-si, Jeollabuk-do, 54896, South Korea}
\email{hwangwon@jbnu.ac.kr}
\thanks{}

\author{Kyunghwan Song}
\address{Department of Mathematics, Jeju National University, 102 Jejudaehakro Jeju, 63243, Republic of Korea}
\curraddr{Department of Mathematics, Jeju National University, 102 Jejudaehakro Jeju, 63243, Republic of Korea}
\email{khsong@jejunu.ac.kr}
\thanks{}

\subjclass[2020]{Primary 11A67, 11A07, 11B65, 05A10, 05A17 }

\date{}

\dedicatory{}


\maketitle


\begin{abstract} 
The greatest integer that does not belong to a numerical semigroup $S$ is called the Frobenius number of $S$, and finding the Frobenius number is called the Frobenius problem.  In this paper,  we solve the Frobenius problem for the numerical semigroups generated by binomial coefficients. As applications, we provide some nontrivial identities among binomial coefficients, and we also connect the main results to the theory of $(s,s+1,s+p)$-core partitions of integers.
\end{abstract}

\maketitle \setcounter{section}{-1}
\section{Introduction}\label{sec_Introduction}
Let $\mathbb{N}$ be the set of nonnegative integers, and suppose that we are given a numerical semigroup $S \subseteq \mathbb{N}$ so that $\mathbb{N} \setminus S$ is finite. Then the Frobenius number $F(S)$ of $S$ is defined to be the largest integer in the finite set $\mathbb{N} \setminus S$.

 It is easy to see that the Frobenius number of $S$ is the greatest integer that is not of the form $\sum_{i=1}^n t_i a_i$ where $t_i \in \mathbb{N}$ and $a_i \in S$ for all $i$ such that $\gcd(a_1, a_2, \cdots, a_n) = 1$. This process is referred to as the Frobenius problem, and the studies related to it have many applications to other fields such as graph theory \cite{Heap (1965),Hujter (1987)} and computer science \cite{Raczunas1996}.

Given two relatively prime numbers $a, b$, the formula for the Frobenius number of the numerical semigroup generated by $a, b$ is expressed in a very easy form \cite{Sylvester1883}, whereas there is no known closed formula for the Frobenius number of numerical semigroups generated by three or more relatively prime integers. In fact, it is known that finding the Frobenius number for three or more relatively prime numbers is a NP-hard problem \cite{Ramirez1996}. Currently, there are only partial results \cite{Robles2012,Tripathi (2017)} or algorithmic methods to compute them \cite{Rodseth1978}.

As some special cases, for the numerical semigroups $S$ generated by three or more relatively prime integers, the research is still ongoing. There are various methods to find the Frobenius number $F(S)$ in this case, and one of the most frequently used methods to compute the Frobenius number is to use the Ap\'ery set of $S$, which is introduced in \cite{Marquez (2015),Rosales2009,Ramirez2009}. In fact, there are many results on the computation of the Frobenius numbers that use the Ap\'ery set \cite{Rosales2015,Rosales2016,Rosales2017}.

Furthermore, there is a generalization of the Frobenius problem, which is called the $p$-Frobenius problem. If $S_p$ is a set of integers that can be represented by more than $p$ ways of the nonnegative integer combinations of elements of a given numerical semigroup, then the $p$-Frobenius number is defined as the largest element of $\mathbb{N}\backslash S_p$, and accordingly, $p$-Ap\'ery set, $p$-genus, etc. can be defined. Note that $0$-Frobenius problem is the classical Frobenius problem. The $p$-Frobenius problem has been actively studied in recent years \cite{Komatsu2022,Komatsu2024,Komatsu2025}.

One of the main motivations for writing this paper is the results of \cite{Rosales2018}. In that paper, the authors solved the Frobenius problem for triangular, tetrahedral, and related numerical semigroups, where all the elements of a given numerical semigroup can be expressed in terms of 
binomial coefficients with fixed lower in common. In other words, the authors studied the numerical semigroups generated by the set
$$
\left\{ \binom{n+m-1}{m}, \cdots, \binom{n+2m-1}{m}\right\}
$$
for each pair of fixed positive integers $n$ and $m$. On the other hand, the research on the numerical semigroups generated by the binomial coefficients with fixed upper does not exist. Thus, we try to consider such semigroups, which we describe now in more details. 

Let $n$ be a positive integer, and consider the following set
$$
B_n = \left\{\binom{n}{1},\binom{n}{2}, \cdots, \binom{n}{n-1}\right\}.
$$
Ram \cite{Ram1909} proved that 
$$\gcd(B_n) = \begin{cases}
	p & \text{when }n = p^m\text{ is a prime power} \\
	1 & \text{otherwise}
\end{cases}$$
and thus,
$$
S(B_n) = \left\langle\left\{\binom{n}{1},\binom{n}{2}, \cdots, \binom{n}{n-1}\right\}\right\rangle
$$
is a numerical semigroup when $n$ is not a prime power, and if $n = p^m$ is a prime power, then
$$
S(B_n) = \left\langle\left\{\frac{1}{p}\binom{n}{1},\frac{1}{p}\binom{n}{2}, \cdots, \frac{1}{p}\binom{n}{n-1}\right\}\right\rangle
$$ 
is a numerical semigroup. In this paper, we solve the Frobenius problem for the numerical semigroup $S(B_n)$ for any integer $n \geq 1.$ Since $\binom{n}{k} = \binom{n}{n-k}$ for each $1 \leq k \leq n$, it is sufficient to describe $S(B_n)$ using only the first $\lfloor \frac{n}{2} \rfloor$ terms where $\lfloor \cdot \rfloor$ denotes the Gauss function. In this regard, some of the main results are summarized in the following theorem. 
\begin{theorem}\label{main thm summary}
    \begin{enumerate}
        \item Let $n = p_1^{k_1}\cdots p_t^{k_t}$ where $p_1< \cdots < p_t$ (with $t \geq 2$) are primes, and $k_i$ are nonnegative integers for each $1 \leq i \leq t$. Then 
    \begin{enumerate}[label=(\alph*)]
        \item The Ap$\acute{e}$ry set of $S(B_n)$ is given by
	$$
		\text{Ap}(S(B_{n}),n) = \Big\{\sum_{i=1}^{t}\sum_{j=1}^{k_i} c_{i,j}\binom{n}{p_i^{j}} ~|~ c_{i,j} \in \mathbb{Z}, 0\leq c_{i,j} \leq p_i-1\Big\}.
	$$
        \item The Frobenius number of $S(B_{n})$ is equal to
	$$
	F(S(B_{n})) =  \sum_{i=1}^{t} \sum_{j=1}^{k_i} (p_i-1)\binom{n}{p_i^j} - n.
	$$
    \end{enumerate}
        \item Let $n = p^m$ where $p$ is a prime and $m \geq 1$ is an integer. Then we have
    \begin{enumerate}[label=(\alph*)]
        \item The Ap$\acute{e}$ry set of $S(B_n)$ is given by
	$$
	\text{Ap}(S(B_{n}),p^{m-1}) = \Big\{\frac{1}{p}\sum_{i=1}^{m-1} c_i\binom{n}{p^i} ~|~ c_{i} \in \mathbb{Z}, 0\leq c_{i} \leq p-1\Big\}.
	$$
    \item The Frobenius number of $S(B_{n})$ is equal to
	$$
	F(S(B_{n})) =  \frac{p-1}{p}\sum_{i=1}^{m-1} \binom{p^m}{p^i} - p^{m-1}.
	$$
    \end{enumerate}
    \end{enumerate}
\end{theorem}
\begin{remark}\label{Erdos rmk}
Theorem \ref{main thm summary}-(1)-(b) solves a problem of Erdos \cite{Erdos 435}. 
\end{remark}
To prove the above theorem, we first compute the Ap$\acute{e}$ry set for the numerical semigroup $S(B_n)$ using a precise description of $S(B_n)$, and then we obtain the Frobenius number for $S(B_n)$ with the help of a well-known relation between the Ap$\acute{e}$ry set and the Frobenius number of a numerical semigroup. For a detailed proof of Theorem \ref{main thm summary}, see Lemma \ref{lem:modn} and Theorems \ref{thm:genemb} and \ref{thm:power} below.

 Note that, in \cite{Rosales2018}, the authors study the numerical semigroups generated by binomial coefficients whose lowers are fixed. On the other hand, we fix the total elements, or the uppers of the binomial coefficients.

This paper is organized as follows: In Section \ref{sec:pre}, we introduce some known definitions and properties that are related to the Frobenius problem of a given numerical semigroup. In Section \ref{sec:bin}, we give some properties of the binomial coefficients with a fixed upper. In Section \ref{sec:main}, we begin with giving a set of binomial coefficients whose nonnegative integer combination is equal to $\binom{n}{m}$ for each $n \geq m \geq 1$. Then, we find a minimal system of generators of $S(B_n)$ and obtain the embedding dimension of $S(B_n)$. By combining these results and the properties obtained in Section \ref{sec:bin}, we provide the Frobenius number, the Ap$\acute{e}$ry set, the genus, and the set of Pseudo-Frobenius number of $S(B_n)$. In Section \ref{sec:appl}, we introduce two applications of our main results, the former of which is about some identities among binomial coefficients with nonnegative integer coefficients, and the latter of which is about the relationship between the main results and $(s,s+1,s+p)$-core partitions.

\section{Preliminaries}\label{sec:pre}
Let $\mathbb{N}$ be the set of nonnegative integers. To begin with, we introduce the notion of a numerical semigroup and a submonoid generated by a nonempty subset of $\mathbb{N}$.
\begin{definition}\label{num semi def} 
	A subset $S$ of $\N$ is called a \emph{numerical semigroup} if $S$ is closed under addition, $0 \in S,$ and the complement $\N \setminus S$ is finite.
\end{definition}
By definition, it is clear that $\mathbb{N}$ is a numerical semigroup.
\begin{definition} \label{def_submonoid}
	Given a nonempty subset $A$ of a numerical semigroup $\mathbb{N}$, we define \emph{the submonoid $\langle A \rangle$  of $(\mathbb{N}, +)$ generated by $A$} as
	\begin{displaymath}
		\langle A \rangle = \{\lambda_1 a_1 + \cdots + \lambda_n a_n ~|~ n \in \mathbb{N}\texttt{\char`\\}\{0\}, a_i \in A, \lambda_i \in \mathbb{N}
	\end{displaymath}
	\textrm{ for all } $i \in \{1,\cdots,n\}\}$.
\end{definition}

In addition, we introduce several theorems and definitions that are related to numerical semigroups and submonoids generated by a nonempty subset of $\mathbb{N}$. The following concepts and facts are already known and the first fact gives a criterion for the submonoid $\langle A \rangle$ generated by a nonempty subset $A \subseteq \mathbb{N}$ to be a numerical subgroup.
\begin{theorem} (\cite{Rosales2015,Rosales2009}).
	Let $\langle A \rangle$ be the submonoid of $(\mathbb{N},+)$ generated by a nonempty subset $A \subseteq \mathbb{N}$. Then $\langle A \rangle$ is a numerical semigroup if and only if $\gcd(A) = 1$.
\end{theorem}
\begin{definition} 
Let $S$ be a numerical semigroup.
\vskip 0.1in
(1) If $S=\langle A \rangle$ for some nonempty subset $A \subseteq \mathbb{N},$ then we say that $A$ is a {\em system of generators of} $S$. 
\vskip 0.1in
(2) If furthermore, we have $S \ne \langle X \rangle $ for any $X \subsetneq A,$ then we say that $A$ is a {\em minimal system of generators of} $S$.
\end{definition}
In terms of the existence of a miminal system of generators of a given numerical semigroup, one has the following interesting result.
\begin{theorem}\label{ros2009} (\cite{Rosales2009}).
	Every numerical semigroup admits a finite and unique minimal system of generators.
\end{theorem}
In light of Theorem \ref{ros2009}, we can speak about the following notion.
\begin{definition} 
Let $S$ be a numerical semigroup. Then the cardinality of the minimal system of generators of $S$ is called the {\em embedding dimension} of $S$ and is denoted by $e(S)$.
\end{definition}
Now, the last condition in Definition \ref{num semi def} gives rise to the following.
\begin{definition}
Let $S$ be a numerical semigroup. Then the cardinality of $\mathbb{N}\texttt{\char`\\}S$ is called the {\em genus of} $S$ and is denoted by $g(S)$.
\end{definition}
The relation among the Frobenius number, genus, and Ap\'{e}ry set of a numerical semigroup is provided in the following lemma.
\begin{lemma} \emph{(\cite{Rosales2009, Selmer1977})}. \label{lem_F_g}
	Let $S$ be a numerical semigroup and let $x \in S\texttt{\char`\\}\{0\}$. Then we have
    \vskip 0.1in
    (a) $F(S) = \max(\text{Ap}(S,x)) - x.$
    \vskip 0.1in
    (b) $g(S) = \frac{1}{x} (\sum_{w \in \text{Ap}(S,x)} w) - \frac{x-1}{2}.$
\end{lemma}
Our last concepts related to a numerical semigroup are given below.
\begin{definition} \label{def_pseudo}
Let $S$ be a numerical semigroup. Then
\vskip 0.1in
(1) An integer $x$ is called a {\em pseudo-Frobenius number} if $x \not\in S$ and $x + s \in S$ for all $s \in S\texttt{\char`\\}\{0\}$. 
\vskip 0.1in
(2) The set of pseudo-Frobenius numbers of $S$ is denoted by $PF(S)$. 
\vskip 0.1in
(3) The cardinality of the set $PF(S)$ is called the {\em type} of $S$, and is denoted by $t(S)$.
\end{definition}
\begin{definition}
    Let $S$ be a numerical semigroup, $a,b \in \text{Ap}(S,x)$, and let $\leq_S$ be a relation such that $a\leq_S b$ if and only if $b-a \in S$. Then there are maximal elements of the set $\text{Ap}(S,x)$ such that 
    \begin{align*}
        & maximals_{\leq S}(\text{Ap} (S, x)) \\
        & = \{w \in \text{Ap}(S, x) ~|~ w' - w \not\in \text{Ap} (S, x) \backslash \{0\}\text{ for all }w' \in \text{Ap} (S, x)\}.
    \end{align*}
\end{definition}
\section{Some identities related to modular congruences for binomial coefficients}\label{sec:bin}
Let $p$ be a prime. In this subsection, the $p$-adic valuation for $\mathbb{N}$ is defined as the function $ \nu_p \colon \mathbb{N} \rightarrow \mathbb{N} \cup \{0\}$
$$
\nu_p(x) = \max\{ \nu \in \mathbb{N} \cup \{0\}: p^{\nu} \mid x\}.
$$
The following congruence which involves certain binomial coefficients and moduli (that is given in terms of the $p$-adic valuation) is useful.
\begin{lemma}\cite{Sun2011}
	Let $p$ be a prime and let $a, m, n \in \mathbb{N}$ with $m \geq n \geq 0$. Then we have
	\begin{equation}\label{eq:modp}
		\frac{\binom{p^am}{p^an}}{\binom{m}{n}} \equiv 1 + [p=2]pn(m-n)\pmod{p^{2+v_p(n)}}
	\end{equation}
	where $[A] = \begin{cases}
		1 &\text{ if }A\text{ is true},\\
		0 &\text{ if }A\text{ is false}.
	\end{cases}$
\end{lemma}
We use this lemma to obtain the following result that is mainly used in this paper.
\begin{lemma}\label{lem:modn}
	Let $n = p_1^{k_1}p_2^{k_2}\cdots p_t^{k_t}$ where $t \geq 1$ and $p_i$ are all distinct primes. Then we have
	$$
	\binom{n}{p^k} \equiv \frac{n}{p^k}\pmod{n}
	$$
	where $p = p_i$ and $k \leq k_i$ for some $i \in \{1,\cdots,t\}$.
\end{lemma}
\begin{proof}
	By substituting $p = 2, a = k \leq k_1, m = \frac{n}{2^k},$ and $n = 1$ in (\ref{eq:modp}), we have
	$$
	\binom{n}{m} = \binom{2^k\cdot \frac{n}{2^k}}{2^k} \equiv \frac{n}{2^k} + 2\cdot n\cdot (\frac{n}{2^k} - 1) \equiv \frac{n}{2^k}\pmod{p^{2+k_1}}.
	$$
	If $p_i$ is an odd prime number, then by substituting $p = p_i, a = k \leq k_i, m = n/p_i^{k},$ and $n = 1$ in (\ref{eq:modp}), we also have
	$$
	\binom{n}{m} = \binom{{p_i}^{k}\cdot \frac{n}{{p_i}^k}}{{p_i}^k} \equiv \frac{n}{{p_i}^k}\pmod{{p_i}^{2+k_i}}.
	$$
	Therefore we have
	$$
	\binom{n}{p_i^{k}} \equiv \frac{n}{{p_i}^k}\pmod{{p_i}^{k_i}}
	$$
	for any prime number $p_i$ and $k \leq k_i$. Since
	$$
	\binom{n}{p_i^{k}} \equiv 0\pmod{p_j^{k_j}}
	$$
	for any $p_j \neq p_i$, the desired result follows from the Chinese Remainder Theorem.
\end{proof}
We introduce the following proposition that is also used in the proof of main result.
\begin{proposition}\label{prop_bound}\cite{Gallier2014}
	Let $p_1, \cdots, p_k$ be $k \geq 2$ integers such that $2\leq p_1\leq \cdots \leq p_k$, with $\gcd(p_1,\cdots,p_k) = 1$. Then, for all $n \geq (p_1 - 1)(p_k -1)$, there exist $i_1, \cdots, i_k \in \mathbb{N}$ such that
	$$
	n = i_1p_1 + \cdots + i_kp_k.
	$$
\end{proposition}
\section{Main Result}\label{sec:main}
In this section, we prove our main result of this paper. To this aim, we begin with the following theorem, which allows us to write a binomial coefficient in terms of other binomial coefficients.

\begin{theorem}\label{thm:general}
	Let $n = p_1^{k_1}\cdots p_t^{k_t}$ for some nonnegative integers $k_i$ and all distinct primes $p_i$. Then for any $m = p_1^{s_1}\cdots p_t^{s_t}$ with $m \leq \lfloor \frac{n}{2}\rfloor$, $\binom{n}{m}$ is a nonnegative integer combination of $\binom{n}{1}, \binom{n}{p_1^{s_1}}, \binom{n}{p_2^{s_2}}, \cdots, \binom{n}{p_t^{s_t}}$.
\end{theorem}
\begin{proof}
	Since $v_{p_i}\left(\binom{n}{m}\right) \geq v_{p_i}\left(\binom{n}{p_i^{s_i}}\right)$ for any $i$, we have
	$$
	\gcd\left(\binom{n}{1}, \binom{n}{p_1^{s_1}},\binom{n}{p_2^{s_2}}, \cdots,\binom{n}{p_t^{s_t}}\right) \mid  \binom{n}{m}.
	$$
	If we let $c = \gcd\left(\binom{n}{1}, \binom{n}{p_1^{s_1}},\binom{n}{p_2^{s_2}}, \cdots,\binom{n}{p_t^{s_t}}\right), x = \min\left(\binom{n}{1}, \binom{n}{p_1^{s_1}},\binom{n}{p_2^{s_2}}, \cdots,\binom{n}{p_t^{s_t}}\right),$ and $y = \max\left(\binom{n}{1}, \binom{n}{p_1^{s_1}},\binom{n}{p_2^{s_2}}, \cdots,\binom{n}{p_t^{s_t}}\right)$, then we have
    $$\gcd\left(\frac{1}{c}\binom{n}{1}, \frac{1}{c}\binom{n}{p_1^{s_1}},\frac{1}{c}\binom{n}{p_2^{s_2}}, \cdots,\frac{1}{c}\binom{n}{p_t^{s_t}}\right) = 1.$$
    It follows from Proposition \ref{prop_bound} that we have
	$$
	F\left(\frac{1}{c}\binom{n}{1}, \frac{1}{c}\binom{n}{p_1^{s_1}},\frac{1}{c}\binom{n}{p_2^{s_2}}, \cdots,\frac{1}{c}\binom{n}{p_t^{s_t}}\right) < \left(\frac{n}{c} - 1\right)\left(\frac{y}{c} - 1\right) < \frac{1}{c}\binom{n}{m}.
	$$
	Since $\binom{n}{m} = c\cdot \frac{1}{c}\binom{n}{m}$, this completes the proof.
\end{proof}

Now we are ready to solve the Frobenius problem for $S(B_n)$ when $n$ is not a prime power.
\begin{theorem}\label{thm:genemb}
	Let $n = p_1^{k_1}\cdots p_t^{k_t}$ where $p_1< \cdots < p_t$ (with $t \geq 2$) are primes, and $k_i$ are nonnegative integers for each $1 \leq i \leq t$. Also, let
	$$
	S(B_{n}) = \left\langle \left\{\binom{n}{1}, \cdots, \binom{n}{n-1}\right\}\right\rangle.
	$$
	Then we have
    \vskip 0.1in
    (a) $S(B_{n}) = \langle\{\binom{n}{1}, \binom{n}{p_1},\binom{n}{p_1^2},\cdots,\binom{n}{p_1^{k_1}}, \cdots, \binom{n}{p_t},\binom{n}{p_t^2},\cdots,\binom{n}{p_t^{k_t}}\}\rangle$.
    \vskip 0.1in
    (b) The embedding dimension of $S(B_{n})$ is equal to 
    $$e(S(B_{n})) = 1 + \sum_{i=1}^{t} k_i.$$
\end{theorem}
\begin{proof}
	Note that
	$$
	\left\langle\left\{\binom{n}{1},\binom{n}{p_1},\binom{n}{p_1^2},\cdots,\binom{n}{p_1^{k_1}}, \cdots, \binom{n}{p_t},\binom{n}{p_t^2},\cdots,\binom{n}{p_t^{k_t}}\right\}\right\rangle \subseteq S(B_n)
	$$
	because $\binom{n}{p^k} > \binom{n}{m}$ and $v_p(\binom{n}{p^k}) < v_p(\binom{n}{m})$ for any $p \in \{p_1, \cdots, p_t\}, 1\leq m\leq p^k-1,$ and $ k \leq k_i$ when $p = p_i$.
	
	If $p_i \nmid  m$ for any $i$, then we have $n | \binom{n}{m}$, and hence, $\binom{n}{m}$ is a nonnegative integer combination of $\binom{n}{1}$.
	
	If $m = p_1^{s_1}\cdots p_t^{s_t}$ where $m \leq \lfloor \frac{n}{2}\rfloor$, then $\binom{n}{m}$ is a nonnegative integer combination of $\binom{n}{1}, \binom{n}{p_1^{s_1}}, \binom{n}{p_2^{s_2}}, \cdots, \binom{n}{p_t^{s_t}}$ by Theorem \ref{thm:general}.
	
If $\gcd(m,n) > 1$ and $m > \lfloor \frac{n}{2}\rfloor$, then we just apply the property $\binom{n}{m} = \binom{n}{n-m}$ to see that $\binom{n}{m}$ is also a nonnegative integer combination of $$\binom{n}{1}, \binom{n}{p_1^{s_1}}, \binom{n}{p_2^{s_2}}, \cdots, \binom{n}{p_t^{s_t}}.$$

	Combining all these cases, we conclude that a minimal system of generators of $S(B_{n})$ is given by $\{\binom{n}{1}, \binom{n}{p_1},\binom{n}{p_1^2},\cdots,\binom{n}{p_1^{k_1}}, \cdots, \binom{n}{p_t},\binom{n}{p_t^2},\cdots,\binom{n}{p_t^{k_t}}\}$.
    \vskip 0.1in
    This completes the proof of both of parts (a) and (b).
\end{proof}
By Theorem \ref{thm:genemb}, we obtain the minimal generator of $S(B_n)$, and by Lemma \ref{lem:modn} we have the modular congruences of the form $\binom{n}{p_i^j}$ for all $1\leq j\leq k_i, 1\leq i\leq t$. Therefore we have the following.

\begin{corollary}\label{cor:ap_n}
	Let $n = p_1^{k_1}\cdots p_t^{k_t}$ where $p_1< \cdots < p_t$ (with $t \geq 2$) are primes, and $k_i$ are nonnegative integers for each $1 \leq i \leq t$. Then 
    \begin{enumerate}[label=(\alph*)]
        \item The Ap$\acute{e}$ry set of $S(B_n)$ is given by
	$$
		\text{Ap}(S(B_{n}),n) = \Big\{\sum_{i=1}^{t}\sum_{j=1}^{k_i} c_{i,j}\binom{n}{p_i^{j}} ~|~ c_{i,j} \in \mathbb{Z}, 0\leq c_{i,j} \leq p_i-1\Big\}.
	$$
        \item The Frobenius number of $S(B_{n})$ is equal to
	$$
	F(S(B_{n})) =  \sum_{i=1}^{t} \sum_{j=1}^{k_i} (p_i-1)\binom{n}{p_i^j} - n.
	$$
    
        \item The genus of $S(B_{n})$ is equal to
       $$g(S(B_{n})) = \sum_{i=1}^{t}\sum_{j=1}^{k_i} \frac{p_i - 1}{2} \binom{n}{p_i^j} - \frac{n-1}{2}.$$
        \item The set of Pseudo-Frobenius number of $S(B_{n})$ is equal to 
        $$PF(S(B_{n})) = \left\{\sum_{i=1}^{t}\sum_{j=1}^{k_i} (p_i-1)\binom{n}{p_i^j} - n \right\}$$ 
        and $t(S(B_n)) = 1.$
    \end{enumerate}
\end{corollary}
\begin{proof}
	We omit the proof of part (a) because it can be obtained directly, and we give a proof of parts (b), (c) and (d). By Lemma \ref{lem_F_g} and part (a), we have
	$$
	\text{Max}(\text{Ap}(S(B_{n}),n)) = \sum_{i=1}^{t}\sum_{j=1}^{k_i} (p_i-1)\binom{n}{p_i^j}
	$$ 
	and hence, we obtain
	$$
	F(S(B_{n})) = \text{Max}(\text{Ap}(S(B_{n}),n)) - n = \sum_{i=1}^{t}\sum_{j=1}^{k_i} (p_i-1)\binom{n}{p_i^j} - n.
	$$
    Also we have
    \begin{align*}
		g(S(B_{n})) & = \frac{1}{n}\sum_{w \in \text{Ap}(S(B_{n}),n)} w - \frac{n-1}{2} \\
		& = \frac{1}{n} \left( \sum_{i=1}^{t}\sum_{j=1}^{k_i}\frac{p_i(p_i-1)}{2}\binom{n}{p_i^j} \cdot \frac{n}{p_i} \right) - \frac{n-1}{2} \\
        & = \sum_{i=1}^{t}\sum_{j=1}^{k_i} \frac{p_i - 1}{2} \binom{n}{p_i^j} - \frac{n-1}{2}.
	\end{align*}
    \noindent
    Finally, since $\sum_{i=1}^{t} (p_i-1)\binom{n}{p_i^j} - w \in \text{Ap}(S(B_{n},n))$ for any  $w \in \text{Ap}(S(B_{n}),n)$, $\text{maximals}_{\leq S(B_n)}(\text{Ap}(S(B_{n}),n)) = \{\sum_{i=1}^{t}\sum_{j=1}^{k_i} (p_i-1)\binom{n}{p_i^j}\}$, and we obtain the Pseudo-Frobenius number.

    This completes the proof.
\end{proof}
In view of \cite[Definitions 2.1 and 2.3-(1)]{DAnna2014}, since we can compute that $\tau_i = p_i - 1$ for the numerical semigroups generated by $B_n$ when $n = p_1^{k_1}\cdots p_t^{k_t}$, we obtain the following interesting result.

\begin{corollary}\label{cor:tele1}
    The numerical semigroup $S(B_n)$ for $n=p_1^{k_1} \cdots p_t^{k_t}$ is telescopic.
\end{corollary}

Moreover, using \cite[Theorem 2.13]{DAnna2014}, we have
\begin{corollary}\label{cor:ci1}
    The numerical semigroup $S(B_n)$ for $n=p_1^{k_1} \cdots p_t^{k_t}$ is complete intersection.
\end{corollary}
Similarly, we can obtain a result for the case when $n = p^m$ is a prime power. In this case, we let
$$
S(B_n) = \left\langle\left\{\frac{1}{p}\binom{n}{1},\frac{1}{p}\binom{n}{2}, \cdots, \frac{1}{p}\binom{n}{n-1}\right\}\right\rangle,
$$
and because of the similarity of the proof, we omit the proofs here, and just state the main results.
\begin{theorem}\label{thm:power}
	Let $n = p^m$ where $p$ is a prime and $m \geq 1$ is an integer. Let
	$$
	S(B_{n}) = \left\langle \left\{\frac{1}{p}\binom{n}{1}, \cdots, \frac{1}{p}\binom{n}{n-1}\right\}\right\rangle.
	$$
	Then we have $S(B_{n}) = \langle\{\frac{1}{p}\binom{n}{1}, \frac{1}{p}\binom{n}{p},\frac{1}{p}\binom{n}{p^2},\cdots, \frac{1}{p}\binom{n}{p^{m-1}}\}\rangle$, and the embedding dimension of $S(B_n)$ is $e(S(B_n)) = m$.
\end{theorem}

By Theorem \ref{thm:power} and Lemma \ref{lem:modn}, we have the following result.
\begin{corollary}\label{thm:power}
	Let $n = p^m$ where $p$ is a prime and $m \geq 1$ is an integer. Then we have
    \begin{enumerate}[label=(\alph*)]
        \item The Ap$\acute{e}$ry set of $S(B_n)$ is given by
	$$
	\text{Ap}(S(B_{n}),p^{m-1}) = \Big\{\frac{1}{p}\sum_{i=1}^{m-1} c_i\binom{n}{p^i} : c_{i} \in \mathbb{Z}, 0\leq c_{i} \leq p-1\Big\}.
	$$
    \item The Frobenius number of $S(B_{n})$ is equal to
	$$
	F(S(B_{n})) =  \frac{p-1}{p}\sum_{i=1}^{m-1} \binom{p^m}{p^i} - p^{m-1}.
	$$
    \item The genus of $S(B_{n})$ is equal to
    $$g(S(B_{n})) = \frac{p-1}{2p}\sum_{i=1}^{m-1}\binom{p^m}{p^i} - \frac{p^{m-1}-1}{2}.$$
    \item The set of Pseudo-Frobenius number of $S(B_{n})$ is equal to
    $$PF(S(B_{n})) = \left\{\frac{p-1}{p}\sum_{i=1}^{m-1} \binom{p^m}{p^i} - p^{m-1} \right \}$$ 
    and $t(S(B_n)) = 1$.
    \end{enumerate}
\end{corollary}

By similar arguments as in Corollaries \ref{cor:tele1} and \ref{cor:ci1}, we have the following result.
\begin{corollary}\label{cor:tele2}
    The numerical semigroup $S(B_n)$ for the case $n = p^m$ is telescopic, and hence, it is complete intersection. 
\end{corollary}


\begin{remark}
    By Corollaries \ref{cor:ap_n} and \ref{thm:power}, for any integer $n \geq 1,$ the numerical semigroup $S(B_n)$ is telescopic, and the type of $S(B_n)$ is equal to $1.$ This implies that $S(B_n)$ is also symmetric, and hence, we have
    \begin{equation}\label{genus eqn}
    g(S(B_n)) = \frac{F(S(B_n)) + 1}{2}.
    \end{equation}
    We computed these values directly using
    $$
    g(S(B_n)) = \frac{1}{n} \sum_{w \in Ap(S(B_n),n)} w - \frac{n-1}{2}
    $$
    and
    $$
    F(S(B_n)) = \max(\text{Ap}(S(B_n),n)) - n
    $$
    for each $n \geq 1$, and it turns out that the above equality (\ref{genus eqn}) holds true.
\end{remark}

\section{Applications}\label{sec:appl}
In this section, we give two applications of our main results.

\subsection{Formulas for binomial coefficients with fixed number of total elements}
As the second application of our main results, we provide some formulas for nonnegative integer combinations of binomial coefficients with fixed number of total elements.
We start with the simplest case: $n = pq$.
\begin{corollary}
	Let $p$ and $q$ be distinct primes. Then for any integer $0 \leq r \leq pq$, there exists an ordered triple $(a,b,c)$ of nonnegative integers such that $\binom{pq}{r} = a\binom{pq}{1} + b\binom{pq}{p} + c\binom{pq}{q}$. In particular, we have the following congruence relation
	$$
	\binom{pq}{r} \equiv bq + cp\pmod{pq}.
	$$
    
\end{corollary}
\begin{proof}
    This follows from Lemma \ref{lem:modn}.
\end{proof}
Therefore, we have the following.
\begin{proposition}\label{prop:applpq}
	Let $p$ and $q$ be distinct primes. Then for any integer $0 \leq r \leq pq$, that is not a multiple of $p$ and $q$, we have 
	\begin{equation}\label{pq_eq}
	\binom{pq}{r} = \frac{\binom{pq}{r} - p\binom{pq}{p} - q\binom{pq}{q}}{pq} \binom{pq}{1} + p\binom{pq}{p} + q\binom{pq}{q}.
	\end{equation}
 \end{proposition}
    Note that if $\binom{pq}{r} \geq p\binom{pq}{p} + q\binom{pq}{q}$ for some $0 \leq r \leq pq,$ then (\ref{pq_eq}) is a representation of $\binom{pq}{r}$ as a nonnegative integer combination of $\binom{pq}{1},\binom{pq}{p}, \text{ and }\binom{pq}{q}$.
    Furthermore, we can characterize such $r$ as follows: without loss of generality, assume that $p \leq q-2$. If $p = 2$, then such $r$ does not exist. If $p \geq 3$, then we have
    \begin{align*}
        p\binom{pq}{p} + q\binom{pq}{q} & \leq (q-2)\binom{pq}{q-2} + q\binom{pq}{q} \\
        & < (2q-2)\binom{pq}{q}.
    \end{align*}
    As an example, let us see the worst case, namely, when $p = 3$. Since 
    $$
        \frac{(2q)(2q-1)\cdots \left(\frac{3}{2}(q-1) + 3\right)}{\left(\frac{3}{2}(q-1) + 1\right) \cdots (q+1)} > \frac{(2q)(2q-1)\cdots \left(\lfloor\frac{7}{4}(q-1)\rfloor + 4\right)}{\left((3q-3) -  \lfloor\frac{7}{4}(q-1)\rfloor\right) \cdots (q+1)}
    $$
    and
    \begin{align*}
        \frac{5}{7}\left(\lfloor\frac{7}{4}(q-1)\rfloor + 4\right) & \geq \frac{5}{7}\left(\frac{7}{4}q + \frac{3}{2}\right) \\
        & = \frac{5}{4}q + \frac{15}{14} = (3q-3) - \left(\frac{7}{4}q - \frac{57}{14}\right) \\
        & > (3q-3) - \lfloor \frac{7}{4}(q-1)\rfloor \\
        & \geq \frac{5}{4}q - \frac{5}{4},
    \end{align*} we have 
    $$
    \frac{(2q)(2q-1)\cdots \left(\frac{3}{2}(q-1) + 3\right)}{\left(\frac{3}{2}(q-1) + 1\right) \cdots (q+1)} > \left(\frac{7}{5}\right)^{\frac{q-\frac{5}{4}}{4}}.
    $$
    Note that $\left(\frac{7}{5}\right)^{\frac{q-\frac{5}{4}}{4}} > 2q - 2$ for $q \geq 59$, and hence, for any $59 \leq \alpha \leq \lfloor\frac{1}{2}pq - q\rfloor = \frac{1}{2}(pq - 2q - 1)$, we have $(2q-2)\binom{pq}{q} \leq \binom{pq}{q + \alpha}$. Consequently, for any $q + 59 \leq r \leq \frac{1}{2}(pq - 1)$, we have a representation of $\binom{pq}{r}$ as a nonnegative integer combination of $\binom{pq}{1}, \binom{pq}{p},$ and $\binom{pq}{q}$. If we choose larger $p$, then we can choose $r$ which is closer to $q$.
Finally, we can generalize Proposition \ref{prop:applpq}.
\begin{theorem}
 Let $n = p_1^{k_1}\cdots p_t^{k_t}$ where $p_1< \cdots < p_t$ (with $t \geq 2$) are primes, and $k_i$ are nonnegative integers for each $1 \leq i \leq t$. Then there are nonnegative integer combinations of $\binom{n}{r}$ where $p_i \nmid r$ for any $i$ and $\binom{n}{r} \geq \sum_{i=1}^{t}\sum_{j=1}^{k_i} p_i^{j}\binom{n}{p_{i}^{j}}$ such that
 $$
 \binom{n}{r} = \frac{\binom{n}{r} - \sum_{i=1}^{t}\sum_{j=1}^{k_i} p_i^{j}\binom{n}{p_{i}^{j}}}{n} \binom{n}{1} + \sum_{i=1}^{t}\sum_{j=1}^{k_i} p_i^{j}\binom{n}{p_{i}^{j}}.
 $$
\end{theorem}
Now, we move to the case when $n=p^2$. In view of \ref{thm:power}, we have the following.
\begin{corollary}
    Let $p$ be an odd prime. Then for any integer $0 \leq r \leq p^2$, there exists an ordered pair $(a,b)$ of nonnegative integers such that $$\binom{p^2}{r} = a\binom{p^2}{1} + b\binom{p^2}{p}.$$ In particular, we have the following congruence relation
    $$
    \binom{p^2}{r} \equiv bp\pmod{p^2}.
    $$
\end{corollary}
Therefore, we also have the following result.
\begin{corollary}\label{cor:appl power}
Let $p$ be an odd prime. Then for any integer $0 \leq r \leq p^2$, that is not a multiple of $p$, we have 
	$$
	\binom{p^2}{r} = \frac{\binom{p^2}{r} - \binom{p^2}{p}}{p^2} \binom{p^2}{1} + \binom{p^2}{p}.
	$$
\end{corollary}

We can generalize Corollary \ref{cor:appl power} as in the following theorem.
\begin{theorem}
    Let $p$ be an odd prime and let $m \geq 1$ be an integer. Then for any integer $0 \leq r \leq p^m$, there exists an ordered $m$-tuple $(a_1,a_2,\cdots,a_m)$ of nonnegative integers such that $$\binom{p^m}{r} = \sum_{i=1}^{m} a_i\binom{p^m}{p^{i-1}}.$$ 
    In particular, we have the following congruence relation
    $$
    \binom{p^m}{r} \equiv \sum_{i=2}^{m} a_ip^{m-i+1} \pmod{p^m}.
    $$
    Using this congruence, we have
    \begin{equation}\label{pm_eq}
    \binom{p^m}{r} = \frac{\binom{p^m}{r} - \sum_{i=2}^{m} p^{i-2}\binom{p^m}{p^{i-1}}}{p^m}\binom{p^m}{1} + \sum_{i=2}^{m} p^{i-2}\binom{p^m}{p^{i-1}}.
    \end{equation}
    Note that if $\binom{p^m}{r} \geq \sum_{i=2}^{m} p^{i-2}\binom{p^m}{p^{i-1}}$, then (\ref{pm_eq}) is a representation of $\binom{p^m}{r}$ as a nonnegative integer combination of $\binom{p^m}{1},\cdots, \binom{p^m}{p^{m-1}}$.
\end{theorem}
It is worth mentioning that our results guarantee the existence of a nonnegative integer combination of several binomial coefficients with the same upper, which is intuitively less obvious to obtain.

\subsection{$(s, s+1, s+p)$-core partitions associated with the numerical semigroups $S(B_n)$.}

Let $S$ be a numerical set with Frobenius number $f:=F(S)$. Then in light of \cite[$\S2$]{KN2021}, there is a unique associated partition $\lambda_S$ with $S.$ Let $Hk(\lambda_S)$ be the set of hook lengths of $\lambda_S.$ 

Also, consider the set $A(S):=\{ n \geq 0~|~n+s \in S~\textrm{for all}~s \in S\}$. Then we have $A(S) \subseteq S$, and if $S$ is a numerical semigroup, then we have $A(S)=S.$ 

Now, the aforementioned sets $Hk(\lambda_S)$ and $A(S)$ are related to each other in the following way.
\begin{Thm}\label{hook thm}
    $Hk(\lambda_S)=\Z_{\geq 0} \setminus A(S).$ 
\end{Thm}
\begin{proof}
    For a proof, see \cite[Theorem 4]{KN2021}.
\end{proof}
We also recall that a partition $\lambda$ of an integer is an $s$-core (resp.\ $(s, s+1)$-core, resp.\ $(s, s+1, s+p)$-core) partition if the hook set $Hk(\lambda)$ does not contain a multiple of $s$ (resp.\ $s$ and $s+1$, resp.\ $s$, $s+1$, and $s+p$). In view of Theorem \ref{hook thm}, we have the following.

\begin{Lem}
    $\lambda_S$ is an $s$-core partition if and only if $s \in A(S).$
\end{Lem}
\begin{proof}
    By definition, if $s \in A(S),$ then $ns \in A(S)=\mathbb{N} \setminus HK(\lambda_S)$ for any $n \geq 1$ so that $\lambda_S$ is an $s$-core partition. Conversely, if $\lambda_S$ is an $s$-core partition, then $s \not \in Hk(\lambda_S)=\mathbb{N} \setminus A(S)$ so that $s \in A(S).$
\end{proof}


\begin{Ex}
Let $S=\{0, 1, 3, 4, 7, 9, 10, \cdots \}.$ Then $A(S)=\{0, 9, 10, \cdots \}$, $F(S)=8,$ $\lambda_S = (5,4,4,2),$ and $Hk(\lambda_S) = \{1,2,3,4,5,6,7,8 \}.$  In particular, $\lambda_S$ is not an $s$-core partition for any $1 \leq s \leq 8$, and it is an $s$-core partition for any $s > 8 = F(S).$
\end{Ex}

Similarly, we also obtain the following two results.
\begin{Cor}
$\lambda_S$ is an $(s,s+1)$-core partition if and only if $\{s, s+1\} \subseteq A(S).$    
\end{Cor}

\begin{Cor}\label{main cor}
Let $p \geq 2$ be an integer. Then $\lambda_S$ is an $(s,s+1, s+p)$-core partition if and only if $\{s, s+1, s+p\} \subseteq A(S).$    
\end{Cor}
In particular, if $s > f :=F(S),$ then $\lambda_S$ is an $(s,s+1,s+p)$-core partition for any integer $p \geq 2.$

Given a numerical set $S$ and an integer $p \geq 2,$ the following notion tells us that its corresponding partition $\lambda_S$ is not necessarily an $(s,s+1,s+p)$-core partition.   
\begin{Def}
    Let $S$ be a numerical set, and let $f=F(S)$. A pair $(s,p)$ of two positive integers with $p \geq 2$ is said to be \emph{admissible for $S$} if $\lambda_S$ is an $(s,s+1, s+p)$-core partition with $s+p < f $.
\end{Def}
The following four examples explain the situations in which there might be no admissible pairs $(s,p)$ for specific numerical sets.
\begin{Ex}
$S=\{0, 1, 3, 4, 7, 9, 10, \cdots \}.$ Then $f=F(S)=8,$ $A(S)=\{0, 9, 10, \cdots \}$, $\lambda_S = (5,4,4,2),$ and $Hk(\lambda_{S}) = \{1,2,3,4,5,6,7,8 \}.$ Then there is no admissible pair for $S$ because if a pair $(s,p)$ of positive integers with $p \geq 2$ satisfies $s+p <f=8,$ then $s < 6$ so that any pair $(s, s+1, s+p)$ cannot be contained in the set $A(S).$
\end{Ex}
\begin{Ex}
    Let $m \geq 1 $ be an integer, and let $S_m = \langle m, m+1, m+2, \cdots, 2m-1 \rangle$ be the numerical semigroup generated by the set $\{m,m+1,m+2,\cdots, 2m-1\}.$ Then we have 
    $$S_m = \{0,m,m+1,m+2, \cdots \}$$ 
    so that $f_m := F(S_m)= m-1$, $\lambda_{S_m} = \underbrace{(1, 1, \cdots,1)}_{(m-1)-\text{times}},$ and $Hk(\lambda_{S_m}) = \{1,2,\cdots, m-1 \}.$ If $\lambda_{S_m}$ is an $s$-core at least, then $s \geq m,$ and so, $s+p > m > f_m $ for any $p \geq 2,$ and hence, there is no admissible pair for $S_m$ for any $m \geq 1.$
\end{Ex}
\begin{Ex}
    Let $m \geq 3$ be an odd integer, and let 
    $$S_m = \langle 2,m \rangle = \{0,2,4,6, \cdots, m-3, m-1, m, m+1, m+2, \cdots \}$$ 
    so that $f_m := F(S_m ) = m-2, \lambda_{S_m} = (\frac{m-1}{2}, \frac{m-3}{2}, \cdots, 2,1),$ and $Hk(\lambda_{S_m}) = \{1,3,5, \cdots, m-4, m-2 \}$. If $\lambda_{S_{m}}$ is an $(s,s+1)$-core partition, then since $s$ must be an even integer, $s+1$ is necessarily an odd integer, and hence, $s \geq m-1.$ Then $s+p \geq m-1+p \geq m+1 > f_m$ for any $p \geq 2,$ and so there is no admissible pair for ${S_m}.$ 
\end{Ex}

\begin{Ex}
    Let 
    $$S = \{0,12,19,24,28,31,34,36,38,40,42,43,45,46,47,48, \cdots  \}$$ 
    be the well-tempered harmonic semigroup, so that $f=F(S)=44,$ 
    $$\lambda_{S} = (12,10,9,8,7,6,6,5,5,4,4,4,3,3,3,3,2,2,2,2,2,2,1,1,1,1,1,1,1,1,1,1,1),$$ 
    and 
    \begin{align*}
        & Hk(\lambda_S ) \\
        & = \{1,\cdots,11,13,\cdots,18,20,\cdots,23,25,26,27,29,30,32,33,35,37,39,41,44 \}.
    \end{align*} 
    Then $\lambda_S$ is a $(42,43,45)$-core partition, but we have $45 >44$, and hence, $(42,2)$ is not admissible pair for $S.$ Since $s$ must be at least $42,$ this shows that there is no admissible pair for $S.$
\end{Ex}

At this point, one might ask whether there is any numerical set $S$ which admits an admissible pair or not, and the next example shows that there is such an $S$.

\begin{Ex}
    Let $S=\langle 5,7,9 \rangle = \{0, 5, 7, 9, 10, 12, 14, \cdots \}$ so that $f=F(S)=13, \lambda_S = (6,5,3,2,1,1,1,1),$ and $Hk(\lambda_S) = \{1,2,3,4,6,8,11,13 \}.$ Then $\lambda_S$ is a $(9,10,12)$-core partition, and since $12 < 13,$ we conclude that $(9,3)$ is an admissible pair for $S.$ In fact, one can show that $(9,3)$ is the only admissible pair for $S.$
\end{Ex}
\vskip 0.02in

In view of previous examples, we can ask the following somewhat naive question.
\begin{Pb}
Classify numerical sets $S$ which admit an admissible $(s,p)$-core partition for some integers $s \geq 1$ and $p \geq 2$.    
\end{Pb}
As a second application of our main result, we give a partial answer for the above problem.
    
To this aim, we take $S:=S(B_n)$, where $S(B_n)$ (depending on $n$) is defined as in $\S0.$ Recall that $S(B_n)$ is a numerical semigroup, and hence, $A(S(B_n))=S(B_n)$ and $Hk(\lambda_{S(B_n)}) = \Z_{\geq 0} \setminus S(B_n).$ 

Our main result in this subsection is the following. 
\begin{Thm}\label{core main thm}
    Let $n \geq 1$ be an integer. Then there is an admissible pair $(s,p)$ for the numerical semigroup $S(B_n)$ for any $p \geq 2.$
\end{Thm}
\begin{proof}
Let $s_1 \geq 1$ and $p \geq 2$ be two integers. We give a partition of the Ap$\acute{e}$ry set $\text{Ap}(S(B_n), n)$ into three subsets $A_1, A_2,$ and $A_3$ as follows: let 
$$A_1 = \{x \in \text{Ap}(S(B_n),n)~|~x \equiv s_1 \pmod{n} \},$$ 
$$A_2 = \{x \in \text{Ap}(S(B_n),n)~|~x \equiv s_1+1 \pmod{n} \},$$ and 
$$A_3 = \{x \in \text{Ap}(S(B_n),n)~|~x \equiv s_1+p \pmod{n} \}.$$
Then for each $1 \leq i \leq 3,$ we let $m_i = \max A_i,$ and $m=\max \{m_1, m_2, m_3\}.$ Now, we consider the following three cases depending on the congruence of $m$ modulo $n:$
\vskip 0.1in
(i) If $m \equiv s_1 \pmod{n},$ then we have $(s_2, s_2+1, s_2+p):= (m, m+1, m+p) \in S(B_n) \times S(B_n) \times S(B_n)$. Now, if $s_2+p < F(S(B_n)),$ then we are done by Corollary \ref{main cor}. If $s_2+p \geq F(S(B_n)),$ then we further let $(s_3, ss_3+1, s_3+p)= (s_2 - \left( \lfloor \frac{d}{n} \rfloor + 1\right) n, s_2 + 1 - \left(\lfloor \frac{d}{n}\rfloor + 1 \right) n, s_2 + p - \left(\lfloor \frac{d}{n} \rfloor + 1\right) n),$ where $d:= s_2+p-F(S(B_n)).$ Then since we have $s_3+p < F(S(B_n))$ and $(s_3, s_3+1, s_3+p) \in S(B_n) \times S(B_n) \times S(B_n),$ it follows that $(s_3,p)$ is an admissible pair for $S(B_n)$ in this case.
\vskip 0.1in
(ii) If $m \equiv s_1+1 \pmod{n},$ then we have $(s_2, s_2+1, s_2+p):=(m+n-1, m+n, m+n+p-1) \in S(B_n) \times S(B_n) \times S(B_n).$ Then we can proceed as in the proof of case (i) above.
\vskip 0.1in
(iii) If $m \equiv s_1 +p \pmod{n},$ then we have $(s_2, s_2+1, s_2+p):= (m+n-p, m+n-p+1, m+n) \in S(B_n)\times S(B_n) \times S(B_n).$ Then we can proceed as in the proof of case (i) above.

This completes the proof.
\end{proof}
In other words, Theorem \ref{core main thm} provides an infinite family of numerical semigroups that admit an admissible pair.

For the purpose of computation, we suggest an algorithm that computes the number of admissible pairs $(s,p)$ for the numerical semigroup $S(B_n).$ 
\begin{algorithm}[hbt!] 
	\caption{An algorithm that computes triples $s,s+1,s+p$ where $s+p < F(S(B_n))$ and $s,s+1,s+p \in S(B_n)$, and the number of triples $(s,s+1,s+p)$ with $s, s+1, s+p \in S(B_n)$ and $s+p < F(S(B_n))$.} 
	$\bullet$ Input: $n, s, p$\\
	$\bullet$ Output: A triple $(s,s+1,s+p)$ such that $s, s+1, s+p \in S(B_n)$, and the number of triples $(s,s+1,s+p)$ with $s, s+1, s+p \in S(B_n)$ and $s+p < F(S(B_n))$  \\
	\begin{algorithmic}[1]
		\STATE $f \leftarrow F(S(B_n))$
		\STATE \textbf{while }$ w \in \text{Ap}(S(B_{n}),n))$ and $\text{list}[i] = 0$ for some $0\leq i\leq 3$ \textbf{do}
		\STATE \quad $i \leftarrow w \% n$
		\STATE \quad \textbf{if }$i = s$ \textbf{then}
		\STATE \quad \quad $\text{list}[0] \leftarrow s$
		\STATE \quad \quad $\text{list}[1] \leftarrow w$
		\STATE \quad \textbf{else if } $i = s+1$ \textbf{then}
		\STATE \quad \quad $\text{list}[2] \leftarrow w$
		\STATE \quad \textbf{else if } $i = s+p$ \textbf{then}
		\STATE \quad \quad $\text{list}[3] \leftarrow w$
		\STATE \textbf{end while}
		\STATE $\text{maxvalue} \leftarrow \max(\text{list})$
		\STATE $\text{maxindex} \leftarrow i$ where $\text{maxvalue} = \text{list}[i]$
		\STATE \textbf{if }$\text{maxvalue }< f$ \textbf{then}
		\STATE \quad \textbf{if }$\text{maxindex }= 1$ \textbf{then}
		\STATE \quad \quad $\text{list}[2] = \text{maxvalue} + 1$
		\STATE \quad \quad $\text{list}[3] = \text{maxvalue} + p$
		\STATE \quad \textbf{else if }$\text{maxindex }= 2$ \textbf{then}
		\STATE \quad \quad $\text{list}[1] = \text{maxvalue} + (n-1)$
		\STATE \quad \quad $\text{list}[2] = \text{list[1]} + 1$
		\STATE \quad \quad $\text{list}[3] = \text{list[1]} + p$
		\STATE \quad \textbf{else if }$\text{maxindex }= 3$ \textbf{then}
		\STATE \quad \quad $\text{list}[1] = \text{maxvalue} + (n-p)$
		\STATE \quad \quad $\text{list}[2] = \text{list[1]} + 1$
		\STATE \quad \quad $\text{list}[3] = \text{list[1]} + p$
		\STATE $\text{diff} \leftarrow f - \text{list}[3]$
        \STATE \textbf{if }$\text{diff} \leq 0$ \textbf{then}
        \STATE \quad \textbf{while }$ 1 \leq i \leq 3$ \textbf{do}
        \STATE \quad \quad $\text{list}[i] = \text{list}[i] - \left(\lfloor \frac{\text{diff}}{n} \rfloor + 1\right) \cdot n$
        \STATE $\text{diff} \leftarrow f - \text{list}[3]$
		\STATE \textbf{if }$\text{diff} \% n = 0$ \textbf{then}
		\STATE \quad $num \leftarrow \text{diff}$
		\STATE \textbf{else }
		\STATE \quad $num \leftarrow \text{diff} + 1$
		\STATE print(list, num)
	\end{algorithmic}
\end{algorithm}

\begin{Ex}\label{ex1 appl}
    Let $n = 50 = 2\cdot 5^2, s' = 65, p = 6$. By Corollary \ref{cor:ap_n}, we have $F(S(B_{50})) = \binom{50}{2} + 4\binom{50}{5} + 4\binom{50}{5^2} - 50 =505642434227223 $. Then using Algorithm 1 below, we can find that 
    $$(s,s+1,s+p) = (379231827789565, 379231827789566, 379231827789571)$$
    is a triple of minimum numbers such that $s, s+1, s+p \in S(B_{50})$ where $s \equiv s'\pmod{n}$ so that $(379231827789565, 6)$ is an admissible pair for $S(B_{50}).$ Also, it can be calculated that the number of triples $(s,s+1,s+p)$ with $s, s+1, s+p \in S(B_{50})$ and $s+p < F(S(B_{50}))$ is equal to $126410606437653$. This means that the number of admissible pairs $(s,p)$ for $S(B_{50})$ equals $126410606437653$.
\end{Ex}

\begin{Ex}\label{ex2 appl}
    Let $n = 70 = 2\cdot 5\cdot 7, s'= 12,$ and $p = 11$. By Corollary \ref{cor:ap_n}, we have $F(S(B_{70})) = 7241062721$. Then using Algorithm 1 again, we can find that 
    $$(s,s+1,s+p) = (4831407922, 4831407923, 4831407933)$$
    is a triple of minimum numbers such that $s, s+1, s+p \in S(B_{70})$ where $s \equiv s'\pmod{n}$ so that $(4831407922, 11)$ is an admissible pair for $S(B_{70}).$ Also, it can be calculated that the number of triples $(s,s+1,s+p)$ with $s, s+1, s+p \in S(B_{70})$ and $s+p < F(S(B_{70}))$ is equal to $2409654789$. This means that the number of admissible pairs $(s,p)$ for $S(B_{70})$ equals $2409654789$.
\end{Ex}

\begin{Rk}
Given an integer $n \geq 1,$ let $N_n$ be the number of admissible pairs $(s,p)$ for the numerical semigroup $S(B_n).$ By the above two examples, it seems to the authors that $N_n$ behaves somehow randomly. It might be also interesting to obtain either an optimal upper bound for $N_n$ or a closed formula for the exact value of $N_n.$   
\end{Rk}

\begin{Rk}
    For explicit computations in Examples \ref{ex1 appl} and \ref{ex2 appl}, we used a computer with the following status-CPU: Intel(R) i7-12800H, 2.40GHz, RAM: 32GB with Python 3.13 version.
\end{Rk}

\section{acknowledgements}
After this paper was written, Dr. Stijn Cambie let us know about the Erdos problem in Remark \ref{Erdos rmk}. The authors give sincere thanks to him.







\begin{thebibliography}{alpha}





























































\bibitem{Erdos 435}
T. F. Bloom, Erdos Problem $\sharp 435$, https://www.erdosproblems.com/435, last accessed 2025. 10. 01

\bibitem{DAnna2014}
M. D’Anna, V. Micale and A. Sammartano, Classes of complete intersection numerical semigroups, Semigroup Forum \textbf{88} (2014), no.\ 2, 453--467.

\bibitem{Gallier2014} J. Gallier, The Frobenius Coin Problem Upper Bounds on The Frobenius Number, (2014), available at \url{https://www.cis.upenn.edu/~cis5110/Frobenius-number.pdf}, last accessed: 16.07.2025.

\bibitem{Heap (1965)}
B. R. Heap and M. S. Lynn, {On a linear Diophantine problem of Frobenius: an improved algorithm,} {\it Numer. Math.} \textbf{7} (1965), 226--231.

\bibitem{Hujter (1987)}
M. Hujter and B. Vizvari, {The exact solutions to the Frobenius problem with three variables,} {\it J. Ramanujan. Math. Soc.} \textbf{2} (1987), no.\ 2, 117--143.

\bibitem{KN2021}
W. Keith and R. Nath, {Partitions with prescribed hooksets,} arxiv:1011.1945v2, 2010.


\bibitem{Komatsu2022}
T. Komatsu, The Frobenius number for sequences of triangular numbers associated with number of solutions, Ann. Comb. \textbf{26} (2022), no.\ 3, 757--779.

\bibitem{Komatsu2024}
T. Komatsu and V. Laohakosol, The $p$-Frobenius problems for the sequence of generalized repunits, Results Math. \textbf{79} (2024), Paper No.\ 26, 25pp.

\bibitem{Komatsu2025}
T. Komatsu and T. Chatterjee, Frobenius numbers associated with Diophantine triples of $ x^ 2+ 3 y^ 2= z^ 3$, Rev. R. Acad. Cienc. Exactas Fís. Nat. Ser. A Mat. RACSAM \textbf{119} (2025), no.\ 3, Paper No.\ 63, 17pp.


\bibitem{Marquez (2015)}
G. M\'{a}rquez-Campos, I. Ojeda, and J. M. Tornero, {On the computation of the Ap$\acute{e}$ry set of numerical monoids and affine semigroups,} {\it Semigroup Forum} \textbf{91} (2015), 139--158, Springer US. 

\bibitem{Raczunas1996}
M. Raczunas and P. Chrzastowski-Wachtel, {A diophantine problem of Frobenius in terms of the least common multiple,} {\it Discrete Math.} \textbf{150} (1996), no.\ 1-3, 347--357. 

\bibitem{Ram1909}
B. Ram., {Common Factors of $\frac{n!}{m!(n-m)!}$, $(m = 1,2,\cdots, n-1)$,} J. Indian Math. Club \textbf{1} (1909), 39--43.

\bibitem{Ramirez1996}
J. L. Ram\'{i}rez-Alfons\'{i}n, {Complexity of the Frobenius problem,} {\it Combinatorica} \textbf{16} (1996), no.\ 1, 143--147.

\bibitem{Ramirez2009}
J. L. Ram\'{i}rez-Alfons\'{i}n and \O. J. R\o dseth, {Numerical semigroups: Ap$\acute{e}$ry sets and Hilbert series,} {\it Semigroup Forum} \textbf{79} (2009), 323--340. 

\bibitem{Robles2012}
A. M. Robles-P\'{e}rez and J. C. Rosales, {The Frobenius problem for numerical semigroups with embedding dimension equal to three,} {\it Math. Comp.} \textbf{81} (2012), no.\ 279, 1609--1617.

\bibitem{Rosales2018}
A. M. Robles-Pérez and J. C. Rosales, The Frobenius number for sequences of triangular and tetrahedral numbers, J. Number Theory \textbf{186} (2018), 473--492.

\bibitem{Rodseth1978}
\O. J. R\o dseth, {On a linear Diophantine problem of Frobenius,} {\it J. Reine Angew. Math.} \textbf{301} (1978), 171--178.


\bibitem{Rosales2015}
J. C. Rosales, M. B. Branco, and D. Torr\~{a}o, {The Frobenius problem for Thabit numerical semigroups,} {\it J. Number Theory} \textbf{155} (2015), 85--99.

\bibitem{Rosales2016}
J. C. Rosales, M. B. Branco, and D. Torr\~{a}o, {The Frobenius problem for repunit numerical semigroups,} {\it Ramanujan J.} \textbf{40} (2016), no.\ 2, 323--334. 


\bibitem{Rosales2017}
J. C. Rosales, M. B. Branco, and D. Torr\~{a}o, {The Frobenius problem for Mersenne numerical semigroups,} {\it Math. Z.} \textbf{286} (2017), no.\ 1-2, 741--749.

\bibitem{Rosales2009}
J. C. Rosales and P. A. Garc\'{i}a-S\'{a}nchez, {\it Numerical Semigroups}, Springer Science \& Business Media, New York, 2009.




\bibitem{Selmer1977}
E. S. Selmer, {On the linear Diophantine problem of Frobenius,} {\it J. Reine. Angew. Math.} \textbf{293/294} (1977), 1--17. 

\bibitem{Sun2011} Z-W. Sun and R. Tauraso, {On some new congruences for binomial coefficients,} Int. J. Number Theory \textbf{7} (2011), no.\ 3, 645--662.

\bibitem{Sylvester1883}
J. J. Sylvester, {Problem 7382,} {\it The Educational Times and Journal of the College Of Preceptors, New Series} \textbf{36} (1883), 177.

\bibitem{Tripathi (2017)}
A. Tripathi, {Formulae for the Frobenius number in three variables,} {\it J. Number Theory} \textbf{170} (2017), 368--389.









\end{thebibliography}
\end{document}